\documentclass[12pt]{amsart}
\usepackage{amsthm}
\usepackage[english]{babel}
\usepackage[utf8]{inputenc}
\usepackage{times}
\usepackage{enumitem}
\usepackage[T1]{fontenc}
\usepackage{graphicx}
\usepackage{float} 
\usepackage{subfigure}
\usepackage{amscd,amsmath,amsfonts,amstext,amssymb,amsbsy,amsopn,amsthm,eucal}
\usepackage{mathrsfs}
\usepackage{txfonts}
\usepackage[figuresright]{rotating}
\usepackage{longtable,booktabs}
\usepackage{rotating,booktabs}
\usepackage{caption}
\usepackage{multirow}
\usepackage{tensor}
\usepackage[numbers,sort&compress]{natbib}
\usepackage{enumitem,amsmath}
\newlist{assumptions}{description}{1}
\setlist[assumptions]{
  font=\bfseries,
  labelwidth=2.8cm,    
  leftmargin=2cm,    
  align=parleft,
  itemsep=0.5\baselineskip
}

\usepackage[table]{xcolor}
\usepackage{fancyhdr}
\usepackage{color} 
\usepackage{hyperref}
\usepackage{tikz}
\usetikzlibrary{arrows.meta}
\usetikzlibrary{chains}
\usetikzlibrary{patterns}

\hypersetup{
	pdftitle={Talenti-Robin},
	pdfauthor={CW},
	pdfsubject={Eigenvalue},
	pdfkeywords={Bossel-Daners, Isoperimetric inequality, Robin Laplacian, First eigenvalue}
	pdfpagelayout=SinglePage,
	pdfpagemode=UseOutlines,  pdfpagemode=UseOutlines,
	colorlinks,
	linkcolor=[rgb]{0,0,0.7},
	urlcolor=[rgb]{0,0,0.4},
	citecolor=[rgb]{0.4,0.1,0}
}

\newcommand{\ton}[1]{\left(#1\right)}

\newcommand{\abs}[1]{\left|#1\right|}

\newcommand{\AVR}{\operatorname{AVR}(g)}

\newcommand{\snk}{\operatorname{sn_\kappa}}
\newcommand{\ak}{\alpha_\kappa}

\DeclareMathOperator{\Ric}{Ric}

\def\dfrac{\displaystyle\frac}

\newtheorem{prop}{Proposition}[section]
\newtheorem{thm}[prop]{Theorem}

\newtheorem{lem}[prop]{Lemma}
\newtheorem{rem}[prop]{Remark}

\numberwithin{equation}{section}

\begin{document}
	
	
	\baselineskip=17pt
	
	
	\title[$p$-Laplacian]{On the Bossel-Daners inequality for the $p$-Laplacian on complete Riemannian manifolds}
	
	\author{Daguang Chen, Shan Li, Yilun Wei}
	
	\thanks{The authors were supported by NSFC grant No. 11831005 and NSFC-FWO 11961131001.}
	
	\subjclass[2010]{{35J92}, {35P30}, {53C24}, {58C40}}
	\keywords{Bossel-Daners; Isoperimetric inequality; Robin Laplacian; Eigenvalue}

	

\begin{abstract}	
	In this paper, we obtain the Bossel-Daners inequality for the first eigenvalue of the
	 $p$-Laplacian with Robin boundary conditions on 
	 complete Riemannian manifolds with lower Ricci curvature bounds. 
	 Furthermore, we demonstrate that the Bossel-Daners inequality extends to compact submanifolds within complete Riemannian manifolds characterized by positive asymptotic volume ratio and non-negative intermediate Ricci curvature.

\end{abstract}

\maketitle

\section{Introduction}

Let $(M,g)$ be an $n$-dimensional complete Riemannian manifold, and let $\Omega\subseteq M$ be a bounded domain with a smooth boundary. For $1<p<\infty$ and $\beta>0$, we denote the first eigenvalue of the $p$-Laplacian with Robin boundary conditions on $\Omega$ by $\lambda_{p}(\Omega,\beta)$, defined as follows
\begin{equation}\label{eq}
	\left\{
	\begin{array}{ll}
		-\Delta_p u=\lambda_{p}(\Omega,\beta) \abs{u}^{p-2} u,  & \text{in}~ \Omega,\\\\
		\abs{\nabla u}^{p-2} \dfrac{\partial u}{\partial N}+\beta\abs{u}^{p-2}u=0, & \text{on} ~ \partial\Omega,
	\end{array}
	\right.
\end{equation}  
where $N$ denotes the unit outward normal of $\partial\Omega$ and $\Delta_p u =\operatorname{div}\left(\abs{\nabla u}^{p-2}\nabla u\right)$. When $p=2$, Bossel \cite{Bossel86} and Daners \cite{Daners06} proved that, in Euclidean space,
\begin{equation}\label{ineq: BD}
	\lambda_{2}(\Omega,\beta)\geq \lambda_{2}\left(\Omega^\sharp,\beta\right),
\end{equation}
where $\Omega^\sharp$ denotes a ball with the same volume as $\Omega$. In 2007, Daners and Kennedy \cite{DK08} showed that equality holds in \eqref{ineq: BD} if and only if $\Omega$ is isometric to a ball. In 2022, Cheng, Li and the first author \cite{CCL22} proved the Bossel-Daners inequality for $n$-dimensional compact Riemannian manifolds with $\Ric\geq n-1$ and $n$-dimensional hyperbolic space. 
In 2023, Li, Wei and the first author \cite{CLW23} established
the Bossel-Daners inequality for 
complete non-compact Riemannian manifolds with $\Ric\geq 0$. Alvino, Nitsch and Trombetti \cite{ANT23} used Talenti comparison theorem to give an alternative approach to inequality \eqref{ineq: BD} for $2$-dimensional Euclidean spaces. For $p>1$, Bucur and Daners \cite{BD10}, Dai and Fu \cite{DF11} independently proved the Bossel-Daners inequality for the $p$-Laplacian in Euclidean space. These results were subsequently generalized to anisotropic $p$-Laplacian in \cite{DG14,DP24}, and to Hermite operator by Chiacchio and Gavitone in \cite{CG22}. For more insights into the eigenvalue problems with Robin boundary conditions, consult works like \cite{ACH20,Laugesen19,LW21,LWW23,Savo20,CW25}, among others.

In the first part, we focus on $n$-dimensional complete Riemannian manifolds with the following assumptions 
\begin{assumptions}
    \item[Assumption I] 
    \begin{itemize}
        \item[(i)] $(M,g)$ is a complete non-compact Riemannian manifold with $\Ric \geq 0$ and positive asymptotic volume ratio
        \begin{equation*}
		\AVR =\lim\limits_{r\rightarrow +\infty} \dfrac{\abs{B_r}}{\omega_n r^n}>0,
	\end{equation*}
	where $\abs{B_r}$ denotes the volume of geodesic ball in $(M,g)$ with radius $r$, and $\omega_n$ denotes the volume of $n$-dimensional Euclidean ball.
	\item[(ii)] $(M,g)$ is the $n$-dimensional hyperbolic space $\mathbb{H}^n$ with constant sectional curvature $-1$.
   \end{itemize}
    \item[Assumption II] $(M,g)$ is a compact Riemannian manifold and satisfies $\Ric\geq(n-1)\kappa$, where $\kappa=0$, $1$ or $-1$. 
\end{assumptions}
The corresponding model spaces $(M_\kappa, g_\kappa)$ are defined as follows 
\begin{equation}\label{(M_kappa,g_kappa)}
	\left(M_\kappa, g_\kappa\right) =
	\left\{
	\begin{array}{ll}
	\ton{\mathbb{R}^n,g_0}, & \text{for Assumption I (i) }, \\
\ton{\mathbb{H}^n,g_{-1}}, & \text{for Assumption I (ii) }, \\
\ton{\mathbb{S}^n(R_{\kappa}),\tilde{g}}, & \text{for Assumption II },
	\end{array}
	\right.
\end{equation}
where $\tilde{g}$ represents the induced metric of $\mathbb{R}^{n+1}$, and $R_\kappa$ is defined by
\begin{equation}\label{R_kappa}
	R_\kappa =\left\{
	\begin{array}{ll}
		\dfrac{d}{\left(1+n\int_{0}^{\pi}{\sin^{n-1}\theta \,d\theta}\right)^{\frac{1}{n}}-1}, & \text{if}~\kappa=0,\\
		1, & \text{if}~\kappa=1,\\
		\dfrac{1}{C\left(d\right)}, & \text{if}~\kappa=-1,
	\end{array}
	\right.
\end{equation} 
where the diameter of $(M,g)$ is denoted by $d$, and $C(z)$ is the unique positive solution to equation
\begin{equation*}
	x\int_0^z{ \left(\cosh t+x \sinh t\right)^{n-1} \,dt}=\int_0 ^\pi{\sin^{n-1} \theta \,d\theta}.
\end{equation*}
Furthermore, we introduce the definition 
\begin{equation}\label{alpha_kappa}
	\ak=
	\left\{
	\begin{array}{ll}
		\AVR, &  \text{for Assumption I (i)},\\
		1, & \text{for Assumption I (ii)},\\
		\dfrac{\abs{M}}{\abs{M_\kappa}}, & \text{for Assumption II}.
	\end{array}
	\right.
\end{equation}
In these manifolds, the Bossel-Daners inequality for the $p$-Laplacian holds, as stated in Theorem \ref{thm: 1}.

\begin{thm}\label{thm: 1}

	Let $\Omega$ be a bounded domain with smooth boundary in $(M,g)$. Suppose that $1<p<\infty$ and $\beta>0$.
	\begin{itemize}[leftmargin=12pt]
	
		\item  For Assumption I, we have
		\begin{equation*}
			\lambda_{p}(\Omega,\beta)\geq \lambda_{p}\left(\Omega^\sharp,\beta\right),
		\end{equation*}
		where $\Omega^\sharp$ is the geodesic ball in $\left(M_\kappa,g_\kappa\right)$ satisfying $\abs{\Omega}=\ak \abs{\Omega^\sharp}$. Moreover, equality holds if and only if $(M,g)$ is isometric to $\left( M_\kappa, g_\kappa\right)$ and $\Omega$ is isometric to a geodesic ball in $(M_\kappa,g_\kappa)$.

		\item For Assumption $II$, we have
		\begin{equation*}
			\lambda_{p}(\Omega,\beta)\geq \lambda_{p}\left(\Omega^\sharp,\beta\right),
		\end{equation*}
		where $\Omega^\sharp$ is the geodesic ball in $\left(M_\kappa,g_\kappa\right)$ satisfying $\abs{\Omega}=\ak \abs{\Omega^\sharp}$. Moreover, for $\kappa=1$, equality holds if and only if $(M,g)$ is isometric to $\left( M_\kappa, g_\kappa\right)$ and $\Omega$ is isometric to a geodesic ball in $(M_\kappa,g_\kappa)$.
	\end{itemize}
\end{thm}

\begin{rem}
The Bossel-Daners inequality for the $p$-Laplacian in Euclidean space was established in \cite{BD10,DF11}. 
When $p=2$, Theorem~\ref{thm: 1} reduces to the results in \cite{Bossel86,Daners06,CCL22,CLW23}.
\end{rem}

In the second part of this paper, we shift our focus to Riemannian manifolds with non-negative intermediate Ricci curvature. In Riemannian manifolds $(M,g)$, at some point $p\in M $, given a unit tangent vector $X\in T_pM$ and a $k$-dimensional subspace $V\subseteq T_p M$ such that $X$ is orthogonal to $V$ at $p$, the intermediate $k$-Ricci curvature of $(X,V)$ is defined by
\begin{equation*}
	\Ric_k(X,V)=\dfrac{1}{k} \sum_{i=1}^k{\left\langle R(e_i,X)e_i, X \right\rangle},
\end{equation*}
where $R$ denotes the curvature tensor and $\left\{e_i\right\}$ represents an orthonormal basis of $V$. $(M,g)$ is said to have non-negative $k$-Ricci curvature if $\Ric_k(X,V)\geq 0$ for all $X$ and $V$ at every point of $(M,g)$. The notion of $k$-Ricci curvature was introduced by Bishop and Crittenden in \cite{BC64}, and was further studied in \cite{Galloway81,KM18,MW24,Mouille22,RW25,Wang24,Wu87}. We now consider $(n+m)$-dimensional complete, non-compact Riemannian manifolds $(M,g)$  with positive asymptotic volume ratio and non-negative $k$-Ricci curvature, where $k=\min \left\{n-1,m-1\right\}$. Denote 
\begin{equation}\label{theta_{m,n}}
	\theta_{m,n} =
	\begin{cases}
		\operatorname{AVR}(g), & m=1~\text{or}~2,\\
		\operatorname{AVR}(g)\cdot\dfrac{(n+m)\omega_{n+m}}{m\omega_m\omega_n}, & m> 2.
	\end{cases}
\end{equation}
Using arguments analogous to these in Theorem \ref{thm: 1}, it can be established that the Bossel-Daners inequality also holds for compact minimal submanifolds in $(M,g)$ as stated below.

\begin{thm}\label{thm: 2}
Let $(M,g)$ be an $(n+m)$-dimensional complete, non-compact Riemannian manifold with positive asymptotic volume ratio and non-negative $k$-Ricci curvature, where $k=\min \{n-1,m-1\}$. Suppose $\Sigma$ is an $n$-dimensional compact minimal submanifold in $(M,g)$ with boundary $\partial \Sigma$. If $1<p< \infty $ and $\beta>0$, then
	\begin{equation*}
		\lambda_{p}(\Sigma,\beta)\geq \lambda_{p}\left(D,\beta\right),
	\end{equation*}
	where $D$ is an $n$-dimensional flat disk in $(\mathbb{R}^{n+m},g_0)$ satisfying $\abs{\Sigma}=\theta_{m,n} \abs{D}$. Furthermore, when $m=1$ or 2, equality holds if and only if $(M,g)$ is isometric to $(\mathbb{R}^{n+m},g_0)$ and $\Sigma$ is isometric to a flat disk in $(\mathbb{R}^{n+m},g_0)$.

\end{thm}

\begin{rem}
For $\beta=\infty$ and $p = 2$, equation \eqref{eq} can be seen as the problem equipped with the Dirichlet boundary condition; the inequality in Theorem \ref{thm: 2} is also referred to as the Faber-Krahn inequality. In this case, Wei and the first author \cite{CW23} established the Faber-Krahn inequality for submanifolds in manifolds with non-negative sectional curvature.
\end{rem}

The paper is  organized as follows. In Section \ref{Preliminaries}, we review the isoperimetric inequality for various types of manifolds. In Section \ref{The H-functional}, the functional $H_\Omega$ is introduced, and its relation with the first eigenvalue of \eqref{eq} is established. In Section \ref{Analysis in radial cases}, we deal with the radial cases in  model spaces. Finally, the proofs of Theorems \ref{thm: 1} and \ref{thm: 2} are presented in Section \ref{Proof of Theorem 1.1}.

\section{Preliminaries}\label{Preliminaries}

\subsection{Isoperimetric inequality in manifolds with lower Ricci curvature bounds}
Let $(M,g)$ be an $n$-dimensional Riemannian manifold satisfying either Assumption $\mathrm{I}$ or $\mathrm{II}$. Suppose $\Omega\subseteq M$ is a bounded domain with a smooth boundary and let $\Omega^\sharp\subseteq M_\kappa$ be a geodesic ball satisfying $|\Omega|=\ak|\Omega^\sharp|$, where $\ak$ is defined in \eqref{alpha_kappa}. The isoperimetric inequality then holds as follows 
\begin{equation}\label{isoperi}
	\abs{\partial \Omega}\geq \ak \abs{\partial \Omega^\sharp}.
\end{equation}
Moreover, for Assumption I and for Assumption II with $\kappa=1$, equality in \eqref{isoperi} holds if and only if $(M,g)$ is isometric to $(M_\kappa,g_\kappa)$, and $\Omega$ is isometric to a geodesic ball in $(M_\kappa,g_\kappa)$.

For non-compact Riemannian manifolds $(M,g)$, Assumption I (i)  was established by Brendle in \cite{Brendle23} (see also \cite{AFM20,BK23}), while  Assumption I (ii) was verified by Schmidt in \cite{Schmidt43} (see also \cite{Chavel84}).
For compact Riemannian manifolds $(M,g)$ with lower Ricci curvature bounds, B\'{e}rard, Besson and Gallot \cite{BBG85} showed that $(M,g)$ admits the isoperimetric inequality, which improves the result by L{\'e}vy and Gromov in \cite{Gromov07}.

\subsection{Isoperimetric inequality in manifolds with non-negative intermediate Ricci curvature}

Let $(M,g)$ be an $(n+m)$-dimensional complete, non-compact Riemannian manifold with positive asymptotic volume ratio and non-negative $k$-Ricci curvature, where $k=\min\{n-1,m-1\}$. Suppose $\Sigma$ is an $n$-dimensional compact minimal submanifold in $(M,g)$ with boundary $\partial \Sigma$ and let $D$ be an $n$-dimensional flat disk in $(\mathbb{R}^{n+m},g_0)$ satisfying $\abs{\Sigma}=\theta_{m,n} \abs{D}$, where $\theta_{m,n}$ is defined in \eqref{theta_{m,n}}. Then, the following inequality holds
\begin{equation}\label{isoperi1}
	\abs{\partial \Sigma}\geq \theta_{m,n} \abs{\partial D}.
\end{equation}
Moreover, for $m=1$ or $2$, equality in \eqref{isoperi1} holds if and only if $(M,g)$ is isometric to $(\mathbb{R}^{n+m},g_0)$ and $\Sigma$ is isometric to a flat disk in $(\mathbb{R}^{n+m},g_0)$. 

The isoperimetric inequality \eqref{isoperi1} was proved by Ma and Wu in \cite{MW24} by Alexandrov-Bakelman-Pucci method (see also \cite{Wang24}).

\section{The functional $H_\Omega$}\label{The H-functional}

Let $(M,g)$ be an $n$-dimensional complete Riemannian manifold, and let $\Omega\subseteq M$ be a bounded domain with a smooth boundary. For $1<p<\infty$ and $\beta>0$, the eigenfunction $u$ associated with $\lambda_{p}(\Omega,\beta)$ is positive and satisfies \eqref{eq}. Consequently, for all $\varphi \in W^{1,p}(\Omega)$, the following holds,
\begin{equation}\label{eq: weak}
	\int_\Omega{\abs{\nabla u}^{p-2} \left\langle \nabla u, \nabla \varphi\right\rangle  \,dV_g}+\beta \int_{\partial\Omega}{u^{p-1} \varphi \,d\mu_g}=\lambda_{p}(\Omega,\beta) \int_\Omega{u^{p-1} \varphi \,dV_g}.
\end{equation}
For $\varphi\in L^\infty(\Omega)$, we define the  functional $H_\Omega(t,\varphi)$ as in \cite{Bossel86} (see also \cite{Daners06,BD10,DF11,CCL22,CW25}), 
\begin{equation*}
	H_\Omega(t,\varphi)=\dfrac{1}{\abs{U_t}}\left(\int_{\partial U_t^i}{\abs{\varphi}^{p-1} \,d\mu_g}-(p-1)\int_{U_t}{\abs{\varphi}^p \,dV_g}+\beta\abs{\partial U_t^e}\right),
\end{equation*}
where
\begin{equation*}
	U_t=\left\{x\in \Omega: ~u(x)>t \right\},~\partial U_t^i =\partial U_t\cap \Omega,~\partial U_t^e =\partial U_t\cap \partial\Omega.
\end{equation*}
Assuming $||u||_{L^\infty(\Omega)}=1$, 
the following lemmas show the relations between $\lambda_{p}(\Omega,\beta)$ and the functional $H_\Omega$.
\begin{lem}\label{E=H}
Let $u > 0$ be the eigenfunction associated with $\lambda_{p}(\Omega,\beta)$ as defined in \eqref{eq}. Then
	\begin{equation}\label{eq: E=H}
		\lambda_{p}(\Omega,\beta)=H_\Omega\left(t,\dfrac{\abs{\nabla u}}{u}\right)
	\end{equation}
for almost every $t\in(0,1)$.
\end{lem}
Using integration by parts, we can easily verify the equation \eqref{eq: E=H}. The detailed proof of Lemma \ref{E=H} is the same as the proof presented in \cite{BD10,DF11}.
\begin{lem}\label{E geq H}
	For any $\varphi\in L^\infty(\Omega)$, there exists $\tilde{t}\in \left(0,1\right)$ such that
	\begin{equation*}
		\lambda_{p}(\Omega,\beta)\geq H_\Omega\left(\tilde{t},\varphi\right).
	\end{equation*} 
	Moreover, if for all $t\in \left(0,1\right)$,
	\begin{equation*}
		\lambda_{p}(\Omega,\beta)\leq H_\Omega(t,\varphi),
	\end{equation*}
	then 
	\begin{equation*}
		\frac{\abs{\nabla u}}{u}=\varphi, \quad \text{almost everywhere in}~\Omega.
	\end{equation*}
	
\end{lem}

\begin{proof}
	
	Lemma \ref{E=H} implies that, for almost every $t\in (0,1)$,
	\begin{equation}\label{3.7}
		\begin{aligned}
			\abs{U_t}\left(\lambda_{p}(\Omega,\beta)-H_\Omega(t,\varphi)\right)&=\abs{U_t}\left(H_\Omega\left(t,\dfrac{\abs{\nabla u}}{u}\right)-H_\Omega(t,\varphi)\right)\\
			&=\int_{\partial U_t^i}{\left(\left(\dfrac{\abs{\nabla u}}{u}\right)^{p-1}-\abs{\varphi}^{p-1}\right) \,d\mu_g}-(p-1)\int_{U_t}{\left(\left(\dfrac{\abs{\nabla u}}{u}\right)^p-\abs{\varphi}^p\right) \,dV_g}\\
			&\geq \int_{\partial U_t^i}{\left(\left(\dfrac{\abs{\nabla u}}{u}\right)^{p-1}-\abs{\varphi}^{p-1}\right) \,d\mu_g}-p\int_{U_t}{\dfrac{\abs{\nabla u}}{u} \left(\left(\dfrac{\abs{\nabla u}}{u}\right)^{p-1}-\abs{\varphi}^{p-1}\right) \,dV_g}\\
			&=\int_{\partial U_t^i}{\left(\left(\dfrac{\abs{\nabla u}}{u}\right)^{p-1}-\abs{\varphi}^{p-1}\right) \,d\mu_g}-p\int_{t}^{+\infty}{\int_{\partial U_s^i}{ \left(\left(\dfrac{\abs{\nabla u}}{u}\right)^{p-1}-\abs{\varphi}^{p-1}\right) \,d\mu_g} \, \dfrac{ds}{s}}\\
			&=-t^{1-p}\dfrac{d}{dt}\left(t^p\cdot  \int_{t}^{+\infty}{\int_{\partial U_s^i}{ \left(\left(\dfrac{\abs{\nabla u}}{u}\right)^{p-1}-\abs{\varphi}^{p-1}\right) \,d\mu_g} \, \dfrac{ds}{s}}\right),
		\end{aligned}
	\end{equation}
that is,
	\begin{equation}\label{3.8}
		t^{p-1}\abs{U_t}\left(\lambda_{p}(\Omega,\beta)-H_\Omega(t,\varphi)\right)\geq-\dfrac{d}{dt}\left(t^p\cdot  \int_{t}^{+\infty}{\int_{\partial U_s^i}{ \left(\left(\dfrac{\abs{\nabla u}}{u}\right)^{p-1}-\abs{\varphi}^{p-1}\right) \,d\mu_g} \, \dfrac{ds}{s}}\right).
	\end{equation}
	Integrating both sides of \eqref{3.8} over $\left(0,1\right)$, we obtain
	\begin{equation*}
		\int_{0}^1{t^{p-1}\abs{U_t}\left(\lambda_{p}(\Omega,\beta)-H_\Omega(t,\varphi)\right) \,dt}\geq 0.
	\end{equation*}
	Thus, there exists some $\tilde{t}\in \left(0,1\right)$ such that $\lambda_{p}(\Omega,\beta)\geq H_\Omega\left(\tilde{t},\varphi\right)$.
	
	To prove the rest part, we assume by contradiction that $\abs{\left\{\dfrac{\abs{\nabla u}}{u}\neq \varphi\right\}}> 0$. Thus, for sufficiently small $t$,
	\begin{equation*}
		\abs{\left\{\dfrac{\abs{\nabla u}}{u}\neq \varphi\right\}\cap U_t}> 0.
	\end{equation*}
	Revisiting the derivation in \eqref{3.7}, for sufficiently small $t$,
	\begin{equation*}
		\abs{U_t}\left(\lambda_{p}(\Omega,\beta)-H_\Omega(t,\varphi)\right)>-t^{1-p}\dfrac{d}{dt}\left(t^p\cdot  \int_{t}^{+\infty}{\int_{\partial U_s^i}{ \left(\left(\dfrac{\abs{\nabla u}}{u}\right)^{p-1}-\abs{\varphi}^{p-1}\right) \,d\mu_g} \, \dfrac{ds}{s}}\right).
	\end{equation*}
	Therefore, 
	\begin{equation*}
		\int_{0}^1{t^{p-1}\abs{U_t}\left(\lambda_{p}(\Omega,\beta)-H_\Omega(t,\varphi)\right) \,dt}> 0,
	\end{equation*}
	which implies the existence of $\tilde{t} \in \left(0,1\right)$ such that $\lambda_{p}(\Omega,\beta)>H_\Omega\left(\tilde{t},\varphi\right)$. This result contradicts the fact that $\lambda_{p}(\Omega,\beta)\leq H_\Omega(t,\varphi)$ for all $t\in (0,1)$.

\end{proof}

\section{Properties of the eigenfunction and radially symmetric function on geodesic balls}\label{Analysis in radial cases}

Let $B_R^\kappa$ denote the geodesic ball in $\left(M_\kappa,g_\kappa\right)$ with radius $R$. We may further assume $R<\pi R_\kappa$ when $(M,g)$ is a compact manifold, where $R_\kappa$ is defined in \eqref{R_kappa}. For $1<p<\infty$ and $\beta>0$, the eigenfunction $v>0$ associated with  $\lambda_{p}\left(B_R^\kappa,\beta\right)$ is the radial solution to the following equation 
\begin{equation}\label{eq: radial}
	\left\{
	\begin{array}{ll}
		-\Delta_p v= \lambda_{p}\left(B_R^\kappa,\beta\right) \abs{v}^{p-2}v,  & \text{in}~ B_R^\kappa,\\\\
		\abs{\nabla v}^{p-2} \dfrac{\partial v}{\partial N}+\beta\abs{v}^{p-2}v=0, & \text{on} ~ \partial B_R^\kappa,
	\end{array}
	\right.
\end{equation}
which is equivalent to
\begin{equation}\label{eq: radial-1 dim}
	\left\{
	\begin{array}{ll}
		-\abs{v'}^{p-2}\left[(p-1)v''+(n-1)\dfrac{\zeta_\kappa'}{\zeta_\kappa}v'\right]=\lambda_{p}\left(B_R^\kappa,\beta\right) \abs{v}^{p-2}v,  & \text{in}~ [0,R),\\\\
		\abs{v'(R)}^{p-2}v'(R)+\beta\abs{v(R)}^{p-2} v(R)=0,~v'(0)=0,
	\end{array}
	\right.
\end{equation}
where
\begin{equation}\label{zeta_kappa}
	\zeta_\kappa(r) =\left\{
	\begin{array}{ll}
		\snk(r), & \text{for Assumption I},\\
		R_\kappa\sin\left(R_\kappa^{-1} r\right), & \text{for Assumption II},
	\end{array}
	\right.
\end{equation}
and 
\begin{equation*}\label{sn_kappa}
	\snk(r) =\left\{
	\begin{array}{ll}
		r, & \text{if}~\kappa=0,\\
		\sin r, & \text{if}~\kappa=1,\\
		\sinh r, & \text{if}~\kappa=-1.
	\end{array}
	\right.
\end{equation*}
Therefore, for all $0<r\leq R$,
\begin{equation}\label{eig: radial}	
	\lambda_{p}\left(B_R^\kappa,\beta\right)=\dfrac{1}{\abs{B_r^\kappa}}\left(\int_{\partial B_r^\kappa}{\left(\dfrac{\abs{\nabla v}}{v}\right)^{p-1} \,d\mu_{g_\kappa}}-(p-1)\int_{B_r^\kappa}{\left(\dfrac{\abs{\nabla v}}{v}\right)^p \,dV_{g_\kappa}} \right).	
\end{equation}
Note that $v'(0)=0$. By the equation \eqref{eq: radial-1 dim}, we have
\begin{align*}
    \dfrac{\partial}{\partial r}\ton{\zeta_\kappa^{n-1}(r)\abs{v^\prime(r)}^{p-2}v^\prime(r)}=-\lambda_p\ton{B_R^\kappa,\beta}v^{p-1}\ton{r}\zeta_\kappa^{n-1}(r)<0,
\end{align*}
which implies that $v^\prime(r)<0$ for $r\in(0,R]$, that is, $v\ton{r}$ is strictly decreasing with respect to $r$.

Assuming $||v||_{L^\infty\left(B_R^\kappa\right)}=1$, Lemma \ref{E=H} indicates that equation \eqref{eq: E=H} holds in $\left(M_\kappa,g_\kappa\right)$ for almost every $t\in (0,1)$. Moreover, for all $t\in\ton{0,1}$,
	\begin{equation}\label{E=H: radial} 
		\lambda_{p}\left(B_R^\kappa,\beta\right)=H_B\left(t,\dfrac{\abs{\nabla v}}{v}\right),
	\end{equation}
where $H_B$ is the functional with respect to $v$.
Additionally, we analyze the monotonicity of $\dfrac{\abs{\nabla v}}{v}$ and provide an upper bound for this quantity in $B_R^\kappa$ as follows.

\begin{lem}\label{monotonicity equiv upper bound}
Let $v\ton{r}> 0$ be the eigenfunction associated with $\lambda_{p}\left(B_R^\kappa,\beta\right)$ as defined in \eqref{eq: radial}. If $f\ton{r}=\dfrac{\abs{v^\prime\ton{r}}}{v\ton{r}}$ and $1<p<\infty$, then
\begin{equation*}
    f'\ton{r}>0,\qquad \text{for all $r\in\ton{0,R}$, }
\end{equation*}
and
\begin{equation*}
		f\ton{r}\leq \beta^{\frac{1}{p-1}},\qquad\text{for all $r\in(0,R]$.}
	\end{equation*}

\end{lem}

\begin{rem}
The monotonicity of 
$f(r)$ in Euclidean space was established in \cite{BD10,DF11}, and this monotonicity was extended to space forms in \cite{LW21}. By Gr$\ddot{o}$nwall's inequality, we give an alternative proof of the monotonicity of $f(r)$.
\end{rem}

\begin{proof}
We prove the lemma in the cases of $1<p\leq 2$ and $2<p<\infty$, respectively.

For $1<p\leq 2$, \eqref{eq: radial-1 dim} yields
	\begin{equation}\label{f'}
		\begin{aligned}
			f'&=-\frac{v''}{v}+\left(\frac{v'}{v}\right)^2=\frac{\lambda_{1,\beta,p}\left(B_R^\kappa\right)}{p-1}f^{2-p}-\frac{n-1}{p-1}\frac{\zeta_\kappa'}{\zeta_\kappa} f+f^2.
		\end{aligned}
	\end{equation}
	Since $f>0$ in $(0,R)$, for all $r\in (0,R)$, one has
	\begin{equation*}
		\begin{aligned}
			f''&=\lambda_{1,\beta,p}\left(B_R^\kappa\right)\frac{2-p}{p-1}f^{1-p}f'+\frac{n-1}{p-1}\dfrac{1}{\zeta_\kappa^2} f-\dfrac{n-1}{p-1}\frac{\zeta_\kappa'}{\zeta_\kappa} f'+2ff'\\
			&>\ton{\lambda_{1,\beta,p}\left(B_R^\kappa\right)\dfrac{2-p}{p-1}f^{1-p}-\dfrac{n-1}{p-1}\dfrac{\zeta_\kappa'}{\zeta_\kappa} +2f}f',
		\end{aligned}
	\end{equation*}
	that is,
	\begin{align}\label{f''}
	      f''>Ff', ~~\text{in}~\ton{0,R},
	  \end{align}
   where $F=\lambda_{1,\beta,p}\left(B_R^\kappa\right)\dfrac{2-p}{p-1}f^{1-p}-\frac{n-1}{p-1}\frac{\zeta_\kappa'}{\zeta_\kappa} +2f$. 
	Besides, \eqref{f'} implies
	\begin{equation}\label{f^'(0)}
		\left\{
		\begin{array}{ll}
			f'(0)= 0, & \quad\text{for}~1<p<2,\\
			f'(0)= \dfrac{\lambda_{1,\beta,p}\left(B_R^\kappa\right)}{n}, & \quad \text{for}~p=2.
		\end{array}
		\right.
	\end{equation}
	By Gr$\ddot{\mathrm{o}}$nwall's inequality, together with \eqref{f''} and \eqref{f^'(0)}, it follows that
	\begin{equation*}
	    f'(r)>f'(0)~\exp\ton{-\int_0^rF(s)ds}\geq 0,\qquad \text{for all $r\in\ton{0,R}$}.
	\end{equation*}
	Meanwhile, noticing $f(0)=0$ and $f(R)=\beta^{\frac{1}{p-1}}$. Thus
$f(r)\leq f(R)=\beta^{\frac{1}{p - 1}}$ in $(0,R]$.

	For $2<p<\infty$, define $h=f^{p-1}=\ton{\dfrac{\abs{v'}}{v}}^{p-1}$. By \eqref{eq: radial-1 dim}, we obtain
	\begin{equation}\label{h'}
	    h'=\lambda_{1,\beta,p}\left(B_R^\kappa\right)-\ton{n-1}\frac{\zeta'_\kappa}{\zeta_\kappa}h+\ton{p-1}h^{\frac{p}{p-1}},
	\end{equation}
	and thus $h'(0)=\dfrac{\lambda_{1,\beta,p}\left(B_R^\kappa\right)}{n}$. Differentiating the equation  \eqref{h'} leads to
	\begin{align*}
	    h''=\ton{n-1}\dfrac{1}{\zeta_\kappa^2}h+\ton{ph^{\frac{1}{p-1}}-(n-1)\frac{\zeta'_\kappa}{\zeta_\kappa}}h'>\ton{ph^{\frac{1}{p-1}}-(n-1)\frac{\zeta'_\kappa}{\zeta_\kappa}}h',
	\end{align*}
	that is,
	$$h''>Gh', ~~\text{in}~\ton{0,R},$$
	where $G=ph^{\frac{1}{p-1}}-(n-1)\dfrac{\zeta'_\kappa}{\zeta_\kappa}$. Due to Gr$\ddot{\mathrm{o}}$nwall's inequality, one can infer that for any $r\in(0,R)$,
	$$h'(r)>h'(0)\exp\ton{-\int_0^rG(s)ds}>0.$$
	Thus, $f'>0$ and $f(r)<\beta^{\frac{1}{p-1}}$ for $r\in(0,R)$. The proof of Lemma \ref{monotonicity equiv upper bound} is  complete.
\end{proof}

Let $\Omega$ be a bounded domain with a smooth boundary in $(M,g)$ satisfying $\abs{\Omega}= \ak\abs{B_R^\kappa}$, where $\ak$ is defined in \eqref{alpha_kappa},  and let $\psi\geq 0$ be a  radially symmetric function in $B_R^\kappa$. Define $\varphi:\Omega\rightarrow \mathbb{R}$ by
\begin{equation}\label{varphi}
	\varphi(q) =\psi(r_\kappa(u(q))),
\end{equation}
where $r_\kappa(t) =I_\kappa^{-1}\left(\ak^{-1}\abs{U_t}\right)$,  $I_\kappa(r) =n\omega_n\int_0^r{\zeta_\kappa^{n-1}(w) \,dw}$, and $\zeta_\kappa$ is defined in \eqref{zeta_kappa}.
We now establish the relationships between $\varphi$ and $\psi$.

\begin{lem}\label{varphi equiv psi}
For all $t, l\geq 0$,
	\begin{equation*}
		\abs{U_t\cap \left\{\varphi>l\right\}}=\ak\abs{B_{r_\kappa(t)}^\kappa\cap \left\{\psi>l\right\}}.
	\end{equation*}
	Consequently, for every $p> 0$,
	\begin{equation*}
		||\varphi||_{L^p\left(U_t\right)}^p=\ak||\psi||_{L^p\left(B_{r_\kappa(t)}^\kappa\right)}^p.
	\end{equation*}
\end{lem}

\begin{proof}
	
	By the definition of $r_\kappa$, we firstly compute
	\begin{equation}\label{s_kappa: der}
		\begin{aligned}
			r_\kappa'(t)&=-\dfrac{\ak^{-1}}{I_\kappa'\left(r_\kappa(t)\right)}\int_{\partial U_t^i}{\dfrac{d\mu_g}{\abs{\nabla u}} }.
		\end{aligned}
	\end{equation}
	On the one hand, \eqref{s_kappa: der} implies
	\begin{equation}\label{|varphi>l|}
		\begin{aligned}
			\abs{U_t\cap \left\{\varphi>l\right\}}&=\int_{U_t\cap \left\{\psi\left(r_\kappa(u)\right)>l \right\}}{\, dV_g }\\
			&=\int_t^{+\infty}{\chi_{\left\{\psi\left(r_\kappa\right)>l \right\}}\cdot \int_{\partial U_w^i}{\dfrac{d\mu_g}{\abs{\nabla u}} }  \, dw}\\
			&=-\int_0^{r_\kappa(t)}{\dfrac{\chi_{\left\{\psi>l\right\} }}{r_\kappa' \left(r_\kappa^{-1}(\sigma)\right)}\cdot \int_{\partial U_{r_\kappa^{-1}(\sigma)}^i}{\dfrac{d\mu_g}{\abs{\nabla u}}}  \, d\sigma}\\
			&=\ak\int_0^{r_\kappa(t)}{\chi_{\left\{\psi>l\right\} }\cdot  I_\kappa'(\sigma) \, d\sigma}.
		\end{aligned}
	\end{equation}
	On the other hand, 
	\begin{equation}\label{|psi>l|}
		\begin{aligned}
			\abs{B_{r_\kappa(t)}^\kappa\cap \left\{\psi>l\right\}}&=\int_{B_{r_\kappa(t)}^\kappa\cap \left\{\psi>l\right\}}{ dV_{g_\kappa} }\\
			&=\int_0^{r_\kappa(t)}{\int_{\partial B_\sigma^\kappa}{ \chi_{\left\{\psi>l\right\}} \, d\mu_{g_\kappa}} \, d\sigma }\\
			&=\int_0^{r_\kappa(t)}{\chi_{\left\{\psi>l\right\} }\cdot  I_\kappa'(\sigma) \, d\sigma}.
		\end{aligned}
	\end{equation}
	Combining \eqref{|varphi>l|} and \eqref{|psi>l|}, we obtain
	\begin{equation*}
		\abs{U_t\cap \left\{\varphi>l\right\}}=\ak \abs{B_{r_\kappa(t)}^\kappa\cap \left\{\psi>l\right\}}.
	\end{equation*}
	Therefore,
	\begin{equation*}
		\begin{aligned}
			||\varphi||_{L^p\left(U_t\right)}^p&=p\int_0^{+\infty}{\abs{U_t\cap \left\{\varphi>l\right\}} \cdot l^{p-1} \, dl}\\
			&= \ak\cdot p\int_0^{+\infty}{\abs{B_{r_\kappa(t)}^\kappa\cap \left\{\psi>l\right\}}\cdot l^{p-1} \, dl}=\ak||\psi||_{L^p\left(B_{r_\kappa(t)}^\kappa\right)}^p.
		\end{aligned}
	\end{equation*}
	
\end{proof}

\section{Proofs of Theorem \ref{thm: 1} and Theorem \ref{thm: 2}}\label{Proof of Theorem 1.1}

\begin{proof}[Proof of Theorem \ref{thm: 1}]
	Let $u$ and $v$ be the eigenfunctions associated with $\lambda_{p}(\Omega,\beta)$ and $\lambda_{p}\left(\Omega^\sharp,\beta\right)$, respectively, where $\Omega^\sharp$ denotes the geodesic ball in $\left(M_\kappa, g_\kappa\right)$ satisfying $\abs{\Omega}=\ak \abs{\Omega^\sharp}$. Assume that $||u||_{L^\infty(\Omega)}=||v||_{L^\infty\left(\Omega^\sharp\right)}=1$. By Lemma \ref{monotonicity equiv upper bound}, we obtain
	\begin{equation}\label{5.2}
		\dfrac{\abs{\nabla v}}{v}\leq \beta^{\frac{1}{p-1}}, \quad \text{in}~ \Omega^\sharp.
	\end{equation}
	Following \eqref{varphi}, define $\varphi:  \Omega\rightarrow \mathbb{R}$ by
	\begin{equation*}
		\varphi(q) =\dfrac{\abs{\nabla v}}{v}(r_\kappa(u(q))),
	\end{equation*}
	where $r_\kappa(t) =I_\kappa^{-1}\left(\ak^{-1}\abs{U_t}\right)$ and $I_\kappa(r) =n\omega_n\int_0^r{\zeta_\kappa^{n-1}(w) \,dw}$. By Lemma \ref{E geq H}, there exists $\tilde{t}\in \left(0,1\right)$ such that
	\begin{equation}\label{5.3}
		\lambda_{p}(\Omega,\beta) \geq H_\Omega(\tilde{t},\varphi).
	\end{equation}
	Combining \eqref{isoperi}, Lemma \ref{varphi equiv psi} and \eqref{5.2}, we obtain
	\begin{equation*}
		\begin{aligned}
			H_\Omega(\tilde{t},\varphi)&=\dfrac{1}{\abs{U_{\tilde{t}}}}\left(\int_{\partial U_{\tilde{t}}^i}{\abs{\varphi}^{p-1} \,d\mu_g}-(p-1)\int_{U_{\tilde{t}}}{\abs{\varphi}^p \,dV_g}+\beta\abs{\partial U_{\tilde{t}}^e}\right)\\
			&\geq \dfrac{1}{\abs{U_{\tilde{t}}}}\left[\int_{\partial U_{\tilde{t}}^i}{\left(\dfrac{\abs{\nabla v}}{v}\Bigg|_{\partial B_{r_\kappa\left(\tilde{t}\right)}^\kappa}\right)^{p-1} \,d\mu_g}-(p-1)\ak\int_{B_{r_\kappa\left(\tilde{t}\right)}^\kappa}{\left(\dfrac{\abs{\nabla v}}{v}\right)^p \,dV_{g_\kappa}}+\left(\dfrac{\abs{\nabla v}}{v}\Bigg|_{\partial B_{r_\kappa\left(\tilde{t}\right)}^\kappa}\right)^{p-1}  \abs{\partial U_{\tilde{t}}^e}\right]\\
			&=\dfrac{1}{\abs{U_{\tilde{t}}}}\left[\left(\dfrac{\abs{\nabla v}}{v}\Bigg|_{\partial B_{r_\kappa\left(\tilde{t}\right)}^\kappa}\right)^{p-1} \abs{\partial U_{\tilde{t}}^i}-(p-1)\ak\int_{B_{r_\kappa\left(\tilde{t}\right)}^\kappa}{\left(\dfrac{\abs{\nabla v}}{v}\right)^p \,dV_{g_\kappa}}+\left(\dfrac{\abs{\nabla v}}{v}\Bigg|_{\partial B_{r_\kappa\left(\tilde{t}\right)}^\kappa}\right)^{p-1}  \abs{\partial U_{\tilde{t}}^e}\right]\\
			&=\dfrac{1}{\abs{U_{\tilde{t}}}}\left[\left(\dfrac{\abs{\nabla v}}{v}\Bigg|_{\partial B_{r_\kappa\left(\tilde{t}\right)}^\kappa}\right)^{p-1} \abs{\partial U_{\tilde{t}}}-(p-1)\ak\int_{B_{r_\kappa\left(\tilde{t}\right)}^\kappa}{\left(\dfrac{\abs{\nabla v}}{v}\right)^p \,dV_{g_\kappa}}\right]\\
			&\geq \dfrac{1}{\abs{B_{r_\kappa\left(\tilde{t}\right)}}}\left[\left(\dfrac{\abs{\nabla v}}{v}\Bigg|_{\partial B_{r_\kappa\left(\tilde{t}\right)}^\kappa}\right)^{p-1} \abs{\partial B_{r_\kappa\left(\tilde{t}\right)}^\kappa}-(p-1)\int_{B_{r_\kappa\left(\tilde{t}\right)}^\kappa}{\left(\dfrac{\abs{\nabla v}}{v}\right)^p \,dV_{g_\kappa}}\right]\\
			&=\dfrac{1}{\abs{B_{r_\kappa\left(\tilde{t}\right)}}}\left(\int_{\partial B_{r_\kappa\left(\tilde{t}\right)}^\kappa}{\left(\dfrac{\abs{\nabla v}}{v}\right)^{p-1} \,d\mu_{g_\kappa}}-(p-1)\int_{B_{r_\kappa\left(\tilde{t}\right)}^\kappa}{\left(\dfrac{\abs{\nabla v}}{v}\right)^p \,dV_{g_\kappa}}\right)\\
			&=H_B\left(v\left(r_\kappa\left(\tilde{t}\right)\right),\dfrac{\abs{\nabla v}}{v}\right).
		\end{aligned}
	\end{equation*}
	Equation \eqref{E=H: radial} implies
	\begin{equation}\label{5.4}
		\begin{aligned}
			H_\Omega\left(\tilde{t},\varphi\right)&\geq H_B\left(v\left(r_\kappa\left(\tilde{t}\right)\right),\dfrac{\abs{\nabla v}}{v}\right)=\lambda_{p}\left(\Omega^\sharp,\beta\right).
		\end{aligned}
	\end{equation}
	Thus, by \eqref{5.3} and \eqref{5.4}, it follows that
	\begin{equation*}
		\lambda_{p}(\Omega,\beta) \geq \lambda_{p}\left(\Omega^\sharp,\beta\right).
	\end{equation*}

	If equality holds, we first show that for all $t\in \left(0,1\right)$,
	\begin{equation}\label{E leq H} 
		\lambda_{p}(\Omega,\beta)\leq H_\Omega(t,\varphi).
	\end{equation}
	Suppose, for contradiction, that there exists $t'\in (0,1)$ such that 
	\begin{equation*}
		\lambda_{p}(\Omega,\beta)> H_\Omega\left(t',\varphi\right).
	\end{equation*}
	Replacing $\tilde{t}$ by $t'$ in the procedure above, it follows that
	\begin{equation*}
		\begin{aligned}
			\lambda_{p}(\Omega,\beta)&> H_\Omega\left(t',\varphi\right)\\
			&\geq H_B\left(v\left(r_\kappa\left(t'\right)\right),\dfrac{\abs{\nabla v}}{v}\right)\\
			&=\lambda_{p}\left(\Omega^\sharp,\beta\right),
		\end{aligned}
	\end{equation*} 
	which leads to a contradiction. Thus, \eqref{E leq H} holds, and by Lemma \ref{E geq H}, we deduce
	\begin{equation}\label{5.6}
		\dfrac{\abs{\nabla u}}{u}=\varphi, \quad \text{almost everywhere in}~\Omega.
	\end{equation}
	Consequently, Lemma \ref{E=H} and \eqref{5.6} imply that for almost every $t\in (0,1)$,
	\begin{equation*}
		H_\Omega(t,\varphi)=H_\Omega\left(t,\dfrac{\abs{\nabla u}}{u}\right)=\lambda_{p}(\Omega,\beta)=\lambda_{p}\left(\Omega^\sharp,\beta\right),
	\end{equation*} 
	which implies that \eqref{5.4} holds with equality for almost every $t\in (0,1)$. Revisiting the proof of \eqref{5.4}, it can be deduced that
	\begin{equation}\label{rigid}
		\abs{\partial U_t}=\ak \abs{\partial B_{r_\kappa(t)}^\kappa}, \quad \text{for~almost~every~}t\in (0,1).
	\end{equation}
	The rigidity of the isoperimetric inequality \eqref{isoperi} then implies that, under an isometry $\Psi$, $(M,g)$ is isometric to $\left(M_\kappa,g_\kappa\right)$ and $U_t$ is isometric to geodesic balls in $\left(M_\kappa,g_\kappa\right)$ for almost every $t\in (0,1)$.
	
	For $\ton{M,g}$ satisfying Assumption I, the arguments of the centers of $\Psi\ton{U_t}$, which are independent of $t$, reduce to the results established in \cite{CW25,AC23}. For the equality case under Assumption II, our analysis adopts the arguments from \cite{CW25,AC23}.
	
	For the model space $\mathbb{S}^n(1)$, we denote by $ x(t)$ the center of $\Psi(U_t)$, by $r(t)$ the geodesic radius of $\Psi(U_t)$ and by $\xi(t) \in T_{x(t)}\mathbb{S}^n(1)$ an arbitrary unit tangent vector. Then the position vector
$$
X(t) = \cos(r(t))x(t) + \sin( r(t))\xi(t) \in \partial \Psi(U_t)
$$
satisfies
$$
u(X(t)) = t.
$$
Differentiating both sides with respect to $t$, one has
\begin{align}\label{diff u}
 \left\langle \nabla u(X(t)),  -r'(t) \sin(r(t))x(t) + \cos(r(t))x'(t)+ r'(t) \cos(r(t))\xi(t) + \sin (r(t))\xi'(t) \right\rangle =1.
\end{align}
Noticing that
\begin{align*}
   &\left\langle x(t), x(t) \right\rangle = 1, \quad \left\langle \xi(t), \xi(t) \right\rangle = 1, \quad \left\langle x(t), \xi(t) \right\rangle = 0, \\
   &\langle x'(t), x(t) \rangle = 0, \quad \langle \xi'(t), \xi(t) \rangle = 0, \quad \nabla u(X(t)) = -|\nabla u(X(t))| (-\sin(r(t))x(t) + \cos (r(t))\xi(t)).
\end{align*}
From \eqref{diff u}, we have
    \begin{align}\label{Simp diff u}
-\abs{\nabla u(X(t))}r'(t) - \abs{\nabla u(X(t))} \langle x'(t), \xi(t) \rangle = 1. 
\end{align}
The isoperimetric inequality guarantees that  $\Psi(U_t) $ are geodesic balls in $\mathbb{S}^n(1)$ for almost every $t \in (0, 1) $, thus
$$
\abs{\Psi(U_t)} = \int_0^{r(t)} n\omega_n \sin^{n - 1} s \, ds.
$$
Differentiating both sides with respect to $t$ yields
$$
- \int_{\partial \Psi(U_t)} \dfrac{1}{|\nabla u|} d\mu_g = r'(t) n\omega_n \sin^{n - 1}\ton{r(t)},
$$
that is,
\begin{align}\label{C(t)}
    - \frac{1}{|\nabla u(X(t))|} = r'(t).
\end{align}
Substituting \eqref{C(t)} into \eqref{Simp diff u}, it follows that
$$
\langle x'(t), \xi(t) \rangle = 0.
$$
Since $x'(t) \in T_{x(t)} \mathbb{S}^n(1)$, and $\xi(t) \in T_{x(t)} \mathbb{S}^n(1)$ is chosen arbitrarily, we have
$$
x'(t) = 0,$$
which implies that the center of $\Psi(U_t)$ is independent of $t$. Let $\left\{t_n\right\}$ be a strictly increasing sequence satisfying \eqref{rigid}, with $t_1=0$ and $\lim\limits_{n\rightarrow +\infty}{t_n }=1$. Then
$$
\Psi(\{u > 0\}) = \Psi\left( \bigcup_n U_{t_n} \right) = \bigcup_n \Psi(U_{t_n}),
$$
that is, $\Psi(\{u > 0\})$ is a nested intersection and union of concentric geodesic balls in $ (M_\kappa, g_\kappa) $. Therefore, $\Omega$ is isometric to a geodesic ball in $\left(M_\kappa,g_\kappa\right)$.

\end{proof}

\begin{proof}[Proof of Theorem \ref{thm: 2}]
	The proof is analogous to that of Theorem \ref{thm: 1}, with the modification that the isoperimetric inequality \eqref{isoperi1} is used in place of \eqref{isoperi}.
	
\end{proof}

\providecommand{\bysame}{\leavevmode\hbox
	to3em{\hrulefill}\thinspace}
	
	\bibliographystyle{plain}
	
	\bibliography{p-laplace-ref}

\begin{thebibliography}{10}

\bibitem{AC23}
Paolo Acampora and Emanuele Cristoforoni.
\newblock An isoperimetric result for an energy related to the {$p$}-capacity.
\newblock {\em Atti Accad. Naz. Lincei Rend. Lincei Mat. Appl.},
  34(4):831--844, 2023.

\bibitem{AFM20}
Virginia Agostiniani, Mattia Fogagnolo, and Lorenzo Mazzieri.
\newblock Sharp geometric inequalities for closed hypersurfaces in manifolds
  with nonnegative {R}icci curvature.
\newblock {\em Invent. Math.}, 222(3):1033--1101, 2020.

\bibitem{ANT23}
Angelo Alvino, Carlo Nitsch, and Cristina Trombetti.
\newblock A {T}alenti comparison result for solutions to elliptic problems with
  {R}obin boundary conditions.
\newblock {\em Commun. Pure Appl. Math.}, 76(3):585--603, 2023.

\bibitem{ACH20}
Ben Andrews, Julie Clutterbuck, Daniel Hauer, et~al.
\newblock Non-concavity of the {R}obin ground state.
\newblock {\em Cambridge J. Math.}, 8(2):243--310, 2020.

\bibitem{BK23}
Zolt{\'a}n~M. Balogh and Alexandru Krist{\'a}ly.
\newblock Sharp isoperimetric and {S}obolev inequalities in spaces with
  nonnegative {R}icci curvature.
\newblock {\em Math. Ann.}, 385(3):1747--1773, 2023.

\bibitem{BBG85}
Pierre Honor{\'e}~Marie B{\'e}rard, G{\'e}rard Besson, and Sylvestre Gallot.
\newblock Sur une in{\'e}galit{\'e} isop{\'e}rim{\'e}trique qui
  g{\'e}n{\'e}ralise celle de paul l{\'e}vy-gromov.
\newblock {\em Invent. Math.}, 80(2):295--308, 1985.

\bibitem{BC64}
Richard~L. Bishop and Richard~J. Crittenden.
\newblock {\em Geometry of manifolds}, volume Vol. XV of {\em Pure and Applied
  Mathematics}.
\newblock Academic Press, New York-London, 1964.

\bibitem{Bossel86}
Marie-H\'el\`ene Bossel.
\newblock Membranes \'elastiquement li\'ees: extension du th\'eor\`eme de
  {R}ayleigh-{F}aber-{K}rahn et de l'in\'egalit\'e{} de {C}heeger.
\newblock {\em C. R. Acad. Sci. Paris S\'er. I Math.}, 302(1):47--50, 1986.

\bibitem{Brendle23}
Simon Brendle.
\newblock Sobolev inequalities in manifolds with nonnegative curvature.
\newblock {\em Comm. Pure Appl. Math.}, 76(9):2192--2218, 2023.

\bibitem{BD10}
Dorin Bucur and Daniel Daners.
\newblock An alternative approach to the {F}aber-{K}rahn inequality for {R}obin
  problems.
\newblock {\em Calc. Var. Partial Differ. Equ.}, 37(1):75--86, 2010.

\bibitem{Chavel84}
Isaac Chavel.
\newblock {\em Eigenvalues in {R}iemannian geometry}, volume 115 of {\em Pure
  and Applied Mathematics}.
\newblock Academic Press, Inc., Orlando, FL, 1984.
\newblock Including a chapter by Burton Randol, With an appendix by Jozef
  Dodziuk.

\bibitem{CCL22}
Daguang Chen, Qing-Ming Cheng, and Haizhong Li.
\newblock Faber-{K}rahn inequalities for the {R}obin {L}aplacian on bounded
  domain in {R}iemannian manifolds.
\newblock {\em J. Differ. Equ.}, 336:374--386, 2022.

\bibitem{CLW23}
Daguang Chen, Haizhong Li, and Yilun Wei.
\newblock Comparison results for solutions of {P}oisson equations with {R}obin
  boundary on complete {R}iemannian manifolds.
\newblock {\em Int. J. Math.}, 34(8):2350045,~19~pp, 2023.

\bibitem{CW23}
Daguang Chen and Yilun Wei.
\newblock Comparison results for filtration equations on manifolds via
  {S}chwarz rearrangements.
\newblock {\em Results Math.}, 78(3):81,~20~pp, 2023.

\bibitem{CW25}
Daguang Chen and Yilun Wei.
\newblock Comparison results for relative {R}obin $p$-capacity on complete
  {R}iemannian manifolds.
\newblock {\em J. Differ. Equ.}, 423:765--796, 2025.

\bibitem{CG22}
Francesco Chiacchio and Nunzia Gavitone.
\newblock The {F}aber-{K}rahn inequality for the {H}ermite operator with
  {R}obin boundary conditions.
\newblock {\em Math. Ann.}, 384:1--16, 2022.

\bibitem{DF11}
Qiu-yi Dai and Yu-xia Fu.
\newblock Faber-{K}rahn inequality for {R}obin problems involving
  $p$-{L}aplacian.
\newblock {\em Acta Math. Appl. Sin. Engl. Ser.}, 27(1):13--28, 2011.

\bibitem{Daners06}
Daniel Daners.
\newblock A {F}aber-{K}rahn inequality for {R}obin problems in any space
  dimension.
\newblock {\em Math. Ann.}, 335(4):767--785, 2006.

\bibitem{DK08}
Daniel Daners and James Kennedy.
\newblock Uniqueness in the {F}aber-{K}rahn inequality for {R}obin problems.
\newblock {\em SIAM J. Math. Anal.}, 39(4):1191--1207, 2008.

\bibitem{DG14}
Francesco Della~Pietra and Nunzia Gavitone.
\newblock Faber-{K}rahn inequality for anisotropic eigenvalue problems with
  {R}obin boundary conditions.
\newblock {\em Potential Anal.}, 41(4):1147--1166, 2014.

\bibitem{DP24}
Francesco Della~Pietra and Gianpaolo Piscitelli.
\newblock Sharp estimates for the first {R}obin eigenvalue of nonlinear
  elliptic operators.
\newblock {\em J. Differ. Equ.}, 386:269--293, 2024.

\bibitem{Galloway81}
Gregory~J. Galloway.
\newblock Some results on the occurrence of compact minimal submanifolds.
\newblock {\em Manuscripta Math.}, 35(1-2):209--219, 1981.

\bibitem{Gromov07}
Misha Gromov.
\newblock {\em Metric structures for {R}iemannian and non-{R}iemannian spaces}.
\newblock Modern Birkh\"auser Classics. Birkh\"auser Boston, Inc., Boston, MA,
  english edition, 2007.
\newblock Based on the 1981 French original, With appendices by M. Katz, P.
  Pansu and S. Semmes, Translated from the French by Sean Michael Bates.

\bibitem{KM18}
Christian Ketterer and Andrea Mondino.
\newblock Sectional and intermediate {R}icci curvature lower bounds via optimal
  transport.
\newblock {\em Adv. Math.}, 329:781--818, 2018.

\bibitem{Laugesen19}
Richard~S. Laugesen.
\newblock The {R}obin {L}aplacian-{S}pectral conjectures, rectangular theorems.
\newblock {\em J. Math. Phys.}, 60(12):121507,~pp~31, 2019.

\bibitem{LW21}
Xiaolong Li and Kui Wang.
\newblock First {R}obin eigenvalue of the $p$-{L}aplacian on {R}iemannian
  manifolds.
\newblock {\em Math. Z.}, 298(3-4):1033--1047, 2021.

\bibitem{LWW23}
Xiaolong Li, Kui Wang, and Haotian Wu.
\newblock On the second {R}obin eigenvalue of the {L}aplacian.
\newblock {\em Calc. Var. Partial Differ.}, 62(9):256,~17~pp, 2023.

\bibitem{MW24}
Hui Ma and Jing Wu.
\newblock Sobolev inequalities in manifolds with non-negative intermediate
  {R}icci curvature.
\newblock {\em J. Geom. Anal.}, 34(3):93,~16~pp, 2024.

\bibitem{Mouille22}
Lawrence Mouill{\'e}.
\newblock Torus actions on manifolds with positive intermediate {R}icci
  curvature.
\newblock {\em J. Lond. Math. Soc.}, 106(4):3792--3821, 2022.

\bibitem{RW25}
Philipp Reiser and David~J. Wraith.
\newblock Positive intermediate {R}icci curvature on fibre bundles.
\newblock {\em SIGMA Symmetry Integrability Geom. Methods Appl.}, 21:1--17,
  2025.

\bibitem{Savo20}
Alessandro Savo.
\newblock Optimal eigenvalue estimates for the {R}obin {L}aplacian on
  {R}iemannian manifolds.
\newblock {\em J. Differ. Equ.}, 268(5):2280--2308, 2020.

\bibitem{Schmidt43}
Erhard Schmidt.
\newblock Beweis der isoperimetrischen eigenschaft der kugel im hyperbolischen
  und sph{\"a}rischen raum jeder dimensionenzahl.
\newblock {\em Math. Z.}, 49(1):1--109, 1943.

\bibitem{Wang24}
Kai-Hsiang Wang.
\newblock Optimal transport approach to {M}ichael-{S}imon-{S}obolev
  inequalities in manifolds with intermediate {R}icci curvature lower bounds.
\newblock {\em Ann. Global Anal. Geom.}, 65(1):7,~25~pp, 2024.

\bibitem{Wu87}
Hung-Hsi Wu.
\newblock Manifolds of partially positive curvature.
\newblock {\em Indiana Univ. Math. J.}, 36(3):525--548, 1987.

\end{thebibliography}

\vspace{1cm}

\begin{flushleft}
	Daguang Chen,
	E-mail: dgchen@tsinghua.edu.cn\\
	Shan Li,
	E-mail: lishan22@mails.tsinghua.edu.cn\\
	Yilun  Wei,
	E-mail:	weiyl19@mails.tsinghua.edu.cn\\
	Department of Mathematical Sciences, Tsinghua University, Beijing, 100084, P.R. China 	
	
\end{flushleft}

\end{document}